\documentclass{article}
\usepackage{amsfonts}
\usepackage{amssymb}
\usepackage{amsmath}
\usepackage{amsthm}

\newtheorem{theorem}{Theorem}[section]
\newtheorem{proposition}[theorem]{Proposition}

\newtheorem{lemma}[theorem]{Lemma}

\begin{document}
   \title{Non-colliding Jacobi processes as limits of Markov chains on Gelfand-Tsetlin graph.}
 \date{}
 \author{Vadim Gorin\thanks{e-mail: vadicgor@gmail.com}\\ Moscow State University \\ Independent University of Moscow}
\maketitle

\begin{abstract}
We introduce a stochastic dynamics related to the measures that
arise in harmonic analysis on the infinite--dimensional unitary
group. Our dynamics is obtained as a limit of a sequence of natural
Markov chains on Gelfand-Tsetlin graph.

We compute finite--dimensional distributions of the limit Markov
process, the generator and eigenfunctions of the semigroup related
to this process.

The limit process can be identified with Doob h--transform of a
family of independent diffusions. Space-time correlation functions
of the limit process have a determinantal form.
\end{abstract}

\section*{Introduction}
The present paper originated from harmonic analysis
on the infinite-dimensional unitary group. Decomposition of natural representations
of $U(\infty)$ into irreducible ones leads to a
family $\{P_{z,w}\}$ of probability measures, which depend on the
two parameters $z$ and $w$. These measures live on the
infinite-dimensional domain $\Omega$. The definition of
the measures can be extended to even lager set of values of the
parameters. According to this, we use four subscripts $z$, $w$,
$z'$, $w'$, instead of two.

The measures $P_{z,w}$ and $P_{z,w,z',w'}$ were introduced in the
papers \cite{Olsh} and \cite{BO}. Our goal is to construct and study
stochastic dynamics related to the measures $P_{z,w,z',w'}$.

The structure of measures $P_{z,w,z',w'}$ substantially depends on
whether parameters are integers or not. In the present paper we
consider the former case, the parameters $z$ and $w$ will be
integers. Denote these integers by $p$ and  $q$, respectively. In our
case the support of the measures is a finite-dimensional subset of
$\Omega$ and can be identified with $\mathcal X= [0,1]^{p+q+1}$. The
probability distributions $P_{p,q,z',w'}$ were explicitly computed
in \cite{BO}. They are given by the Jacobi orthogonal polynomial
ensemble.

We construct a family $J_{p,q,z',w'}(t)$ of stationary Markov
processes in $\mathcal X$. Each $P_{p,q,z',w'`}$ serves as an
invariant distribution of the corresponding $J_{p,q,z',w'}(t)$. The
processes $J_{p,q,z',w'}(t)$ are obtained as limits of certain
Markov chains on the Gelfand-Tsetlin graph. These Markov chains
arise in a natural way, due to the approximation of the
infinite-dimensional unitary group by the increasing chain of the
groups $U(N)$. We call these chains ``up--down'' chains. Similar
Markov chains have already appeared earlier, they were studied by
Fulman \cite{Fu1}, \cite{Fu2}, Borodin and Olshanski \cite{BO2} and
Petrov \cite{Pe}

The proof of the convergence of our Markov chains on the Gelfand-Tsetlin
graph is based on the special determinantal form of the transition probabilities of
these Markov chains. We reduce convergence of the transition probabilities to convergence
of certain matrices. Passing to the limit in these matrices is simplified by the fact
that we are able to diagonalize them.

The limit Markov processes $J_{p,q,z',w'}(t)$ turn out to have some
interesting properties. One proves that $J_{p,q,z',w'}(t)$ is a
time-dependent determinantal point process. It means that its dynamical (space-time)
correlation functions have determinantal form and can be expressed
through the minors of certain extended kernel.

We  also explain that $J_{p,q,z',w'}(t)$
can be identified with Doob $h$--transform of $p+q+1$ independent
random motions in $[0,1]$.

We fully describe processes $J_{p,q,z',w'}(t)$, i.e. we compute
their transition probabilities and write down generators and
eigenfunctions of the Markov semigroups corresponding to the
processes.

The study of the harmonic analysis on the infinite-dimensional
unitary group shows numerous connections with infinite symmetric
group $S(\infty)$ (see \cite{KOV} and \cite[Part (m) in Intoduction]{BO} ). The constructions of the present paper are
similar to the constructions of \cite{BO2}, where the dynamics related
to $S(\infty)$ were studied.

It turned out that Markov chains of the present paper have lots of
similarities with the ones considered in \cite{Gor}. We use some
ideas and formulas of \cite{Gor}.

We want to emphasize that while in our case of integral parameters all
processes live in a finite-dimensional space, in the case of
arbitrary parameters the space becomes infinite-dimensional. Thus,
our case can be viewed as a degeneration of the general case. The
problem of constructing dynamics for non-integral parameters remains
open, it seems like one has to use different arguments for that
case.

The paper is organized in the following way. In the first four sections
we do some preparatory work: In Section 1 we
introduce Markov chains related to the
Gelfand-Tsetlin graph. In Section 2 we present the
measures $P_{z,w,z',w'}$ and formulate the problem, solved in the
present paper. In Section 3 we study basic properties of the
Markov chains under consideration. Finally, In Section 4 we express
transition probabilities of the Markov chains in the determinantal
form, which is convenient for limit transitions.

In Section 5 we state and prove the main results of the paper. We prove
the existence of the limit process and compute its one-dimensional
distributions and transition probabilities (see Theorem \ref{theorem_main_result}). Determinantal form
of transition probabilities imply the determinantal property of the limit Markov processes
(see Proposition \ref{proposition_determinantal_kernel}).

In Section 6 we study Markov semigroup related to the constructed process. We compute eigenfunctions
and eigenvalues of the semigroup (see Theorem \ref{Theorem_Semigroup_eigenfunctions}); then we find a simple
expression for the generator of the semigroup (see Theorem \ref{theorem_generator}).

In Section 7 we explain
the connection between our process and Doob $h$--transform.

The author is grateful to G.~Olshanski for suggesting the problem
and numerous fruitful discussions.

The author was partially supported by the Moebius Contest Foundation
for Young Scientists and Leonhard Euler's Fund of Russian
Mathematics Support.

\section{General Markov chains on the Gelfand-Tsetlin graph}

In this section we introduce Markov chains studied in the paper.

\emph{The Gelfand-Tsetlin graph} $\mathbb {GT}$ (also known as the
graph of signatures) is a graded graph, whose vertices are so-called
\emph{signatures}. The $N$--th level of the graph, denoted by
$\mathbb{GT}_N$, consists of $N$--tuples of integers
$\lambda=(\lambda_1\ge\dots\ge\lambda_N)$ which are are called
signatures of the order $N$. We join two signatures $\lambda \in
\mathbb{GT}_N$ and $\mu\in\mathbb{GT}_{N+1}$ by an edge and write
$\lambda\prec\mu$ if
$$\mu_1\ge\lambda_1\ge\mu_2\ge\dots\ge\lambda_N\ge\mu_{N+1}.$$
We agree that $\mathbb{GT}_0$ consists of a single element, the
empty signature $\emptyset$. $\emptyset$ is joined by an edge with
every signature from $\mathbb{GT}_1$.

By a path in the Gelfand-Tsetlin graph we mean a sequence of
vertices
$$\lambda(n) \prec \lambda(n+1)\prec \dots \prec
\lambda(m),\quad \lambda_i\in\mathbb{GT}_i.$$

Denote by $\operatorname{Dim}(\lambda)$ the number of paths going
from $\emptyset$ to $\lambda\in\mathbb{GT}_N$.

For $\mu\in \mathbb{GT}_{N+1}$ and $\lambda\in\mathbb{GT}_{N}$ set
$$
p^\downarrow(\lambda\mid \mu)=\begin{cases}
\operatorname{Dim}(\lambda)/\operatorname{Dim}(\mu),& \text{if }
\lambda \prec \mu \\ 0,& \text{otherwise.}
\end{cases}
$$

It is clear that
$$
 \sum_{\lambda\in\mathbb{GT}_{N}:\lambda\prec\mu}
 p^\downarrow(\lambda\mid \mu)=1.
$$

The numbers $p^\downarrow(\lambda\mid \mu)$ are called
\emph{cotransition probabilities} or \emph{``down'' transition
function}.

A sequence $\{M_N\}_{N=0,1,\dots}$, where  $M_N$ is a probability
measure on $\mathbb{GT}_N$, is called a \emph{coherent system of
distributions} provided that for any $N\ge 0$ the measures $M_{N+1}$
and $M_{N}$ are consistent with the cotransition probabilities from
$\mathbb{GT}_{N+1}$ to $\mathbb{GT}_{N}$, i.e.
$$
 \sum_{\mu\in\mathbb{GT}_{N+1}}
 p^\downarrow(\lambda\mid \mu)M_{N+1}(\mu)=M_N(\lambda) \quad \text{ for any
 }\lambda\in\mathbb{GT}_N.
$$

Define the \emph{support} of a coherent system $M$ as the subset
$$
supp(M)=\{\lambda\in\mathbb{GT}:M(\lambda)\neq 0\} \subset
\mathbb{GT},
$$
where $M(\lambda)$ means the measure $M_N(\lambda)$ of the singleton
$\{\lambda\}\subset \mathbb{GT}_N$.

Given a coherent system of distributions we may speak about
\emph{transition probabilities} or \emph{``up'' transition
functions}. The transition probabilities $p^\uparrow(\mu\mid
\lambda)$ are defined for all $\lambda\in supp(M)$ by
$$
 p^\uparrow(\mu\mid \lambda)=\frac{M_{N+1}(\mu)}{M_N(\lambda)}\cdot p^\downarrow(\lambda\mid \mu),
 \quad
 \lambda\in\mathbb{GT}_N,\mu\in\mathbb{GT}_{N+1},M_{N}(\lambda)\neq
 0.
$$

Any coherent system of distributions $\{M_N\}$ defines a Markov
chain $H(t), t=0,1,\dots$ on the state set $supp(M)$. $H(t)$ takes
values in $L_t=\mathbb{GT}_t \bigcap supp(M)$ and its one-dimensional distributions
are given by $M_t$. The transition probabilities of $H(t)$ are
precisely the numbers $p^\uparrow(\mu\mid \lambda)$:
$$
 {\rm Prob}\{H(t+1)=\mu\mid H(t)=\lambda\}=p^\uparrow(\mu\mid \lambda).
$$

Note that, while the transition probabilities of $H(t)$ depend on
the coherent system, the cotransition probabilities ${\rm
Prob}\{H(t)=\lambda\mid H(t+1)=\mu\}$ are nothing but the quantities
$p^\downarrow(\lambda\mid \mu)$, and they depend solely on the
structure of the Gelfand-Tsetlin graph.

We call $H(t)$ \emph{the up chain}, corresponding to $\{M_N\}$.

We also define a family of stationary Markov chains $T_N(t),\quad
N=0,1,\dots$. The state space of $T_N(t)$ is $L_N=\mathbb{GT}_N
\bigcap supp(M)$. The distribution of $T_N(t)$ is given by $M_N$.
The transition probabilities are given by the composition of ``up''
and ``down'' transition functions, from $L_N$ to $L_{N+1}$ and then
back to $L_N$. Denote by $p^{st}_N(\lambda\mid \lambda')$ the
transition probabilities of $T_N$. We have
$$
 p^{st}_N(\lambda\mid \lambda')=\sum_{\mu\in L_{N+1}} p^\downarrow(\lambda\mid \mu) p^\uparrow(\mu\mid \lambda').
$$

We call $T_N(t)$ \emph{the $N$-th level up-down chain},
corresponding to $\{M_N\}$.

In the present paper we study the behavior of the up-down chains,
corresponding to certain coherent systems, as $N\to\infty$.

{\bf Remarks.} Our definitions makes sense not only for the
Gelfand-Tsetlin graph $GT$. In their papers Kerov and Vershik introduced transition and cotransition probabilities for
arbitrary branching graphs (see e.g. \cite[Section 9]{KOV} and references therein).
Thus, we can define Markov chains similar to
$H(t)$ and $T_N(t)$ for an arbitrary branching graph. The up-down
chains for the Young graph were recently studied by Borodin and
Olshanski in \cite{BO2}. Our work was influenced by that paper.

\section{Measures $P_{z,w,z',w'}$ and construction of the limit
process}
 In the paper \cite{Olsh} a $4$-parameter family of coherent systems
 $M_N^{z,w,z',w'}$ was introduced. These measures appear in a natural way in harmonic
 analysis on the infinite-dimensional unitary group. They are given by
\begin{multline*}
M_N^{z,w,z',w'}(\lambda)=\left( S_N(z,w,z',w') \right) ^{-1}\cdot
\operatorname{Dim}^2(\lambda) \\ \times \prod\limits_{i=1}^N
\frac{1}{\Gamma(z-\lambda_i+i)\Gamma(w+N+1+\lambda_i-i)\Gamma(z'-\lambda_i+i)\Gamma(w'+N+1+\lambda_i-i)}
,
\end{multline*}
where $z,w,z',w'$ are complex numbers, the quadruple $(z,w,z',w')$
belongs to the set of \emph{admissible values} (see \cite[Definition
7.6]{Olsh}) and $S_N(z,w,z',w')$ is a normalization constant
$$
 S_N(z,w,z',w')=\prod\limits_{i=1}^N \frac{\Gamma(z+w+z'+w'+i)}
   {\Gamma(z+w+i)\Gamma(z+w'+i)\Gamma(z'+w+i)\Gamma(z'+w'+i)\Gamma(i)}.
$$

Let us assume that $z=k$ and $w=l$, where $k,l\in\mathbb Z$, $k+l\ge
0$. It was shown in \cite{Olsh} that $(k,l,z',w')$ forms an
admissible quadruple of parameters if $z'$ and $w'$ are real and
$z'-k>-1$, $w'-l>-1$.

Note that for any integer $n$ the shift
$$ k\mapsto k+n,\quad l\mapsto l-n, \quad z'\mapsto z'+n,\quad w'\mapsto w'-n\quad
\lambda_i\mapsto \lambda_i + n$$ leaves the probability
distributions $M_N$ invariant. This means that essentially the
coherent system $\{M_N^{k,l,z',w'}\}$ depends on three, not four
parameters. From now on we assume that $k=p\ge 0$ and $l=0$; $z>
p-1$ and $w> -1$.

It is easily seen from the definition of $M_N^{p,0,z',w'}$ that
$supp (M)$ consists of the signatures $\lambda\in\mathbb{GT}_N$ such
that
$$
p\ge\lambda_1\ge\dots\ge\lambda_N\ge 0.
$$
We always assume below that $\lambda$ satisfies these inequalities.

Set $X_{N,p}=\{0,1,\dots,N+p-1\}$. Observe that $0\le
\lambda_i-i+N\le N+p-1$. Let us associate to $\lambda$ a collection
$X(\lambda)$ of distinct points  in $X_{N,p}$ as follows
$$
 X(\lambda)=X_{N,p}\setminus\{ \lambda_1-1+N,\lambda_2-2+N,\dots,\lambda_N-N+N \}.
$$
It is clear that $|X(\lambda)|=p$ and $X(\lambda)$ defines $\lambda$
uniquely. Let
$$X(\lambda)=\{0\le x_1<x_2<\dots<x_p\le N+p-1\}.$$

{\bf Remark.}
Equivalently $\lambda\mapsto X(\lambda)$ can be described in the following way.
First, we identify $\lambda$ with a Young diagram. Row lengths of this diagram are
$\lambda_i$. Then we consider the transposed diagram and its row lengths (in other words, we
consider heights of the columns of our diagram). We arrive to a collection
of numbers $$0\le\mu_1\le\mu_2\le\dots\le\mu_p\le N.$$
Finally, we set $x_i=\mu_i+i-1$. Equivalence of this definition with the one given above
follows, e.g. from \cite[(1.7)]{Mac}.

We work with $X(\lambda)$ instead of $\lambda$ by the following reasons. First,
$|X(\lambda)|=p$ and, thus, it does not depend on $N$. Consequently, we can identify
$X(\lambda)$ corresponding to different $N$ with elements of one fixed space. Second,
transition to the next level of the Gelfand-Tsetlinn graph is very simple in the $X(\lambda)$
interpretation. We will provide more details later.

Denote by $P_N^{p,z',w'}$ the pushforward of the probability measure
$M_N^{p,0,z',w'}$ under the map $\lambda \mapsto X(\lambda)$. It is
clear that $P_N$ is a probability measure on $(X_{N,p})^p$ and its
support is
$$
 supp(P_N)=\{ (x_1,\dots,x_p)\in (X_{N,p})^p:x_1<x_2<\dots<x_p\}.
$$

For any $N=0,1,2\dots$ we embed the set $(X_{N,p})^p$ into
$\mathcal X=[0,1]^p$ as follows:
$$ \pi_N : (x_1,\dots,x_p)\mapsto
\left(\frac{x_1}{N+p-1},\dots,\frac{x_p}{N+p-1}\right).$$ Denote by
$\tilde P_N^{p,z',w'}$ the pushforward of the measure
$P_N^{p,z',w'}$ under the embedding $\pi_N$.
\begin{proposition}
 As $N\to\infty$ the measures $\tilde P_N^{p,z',w'}$ weakly
 converge to a measure $\mu_{p,z',w'}$.

 The measure $\mu_{p,z',w'}$ is given by its density function $\rho_{p,z',w'}$
 \begin{multline*}
   \rho_{p,z',w'}(x_1,\dots,x_p)\\
     =\begin{cases}
      B_{p,z'w'}\cdot \prod\limits_{i<j} (x_i-x_j)^2\prod\limits_{i=1}^p
      x_i^{w'} (1-x_i)^{z'-p},& x_1<x_2<,\dots,<x_p,\\
      0,& otherwise,
     \end{cases}
 \end{multline*}
 where $B_{p,z'w'}$ is a normalization constant.
\end{proposition}
\begin{proof}
  A slightly different version of this proposition was proved in
  \cite[Theorem 11.6]{BO}. We will also verify this proposition
  by straightforward computations later (see Proposition \ref{proposition_density_lim}).
\end{proof}

The aim of the present paper is to introduce a stationary stochastic
process, which has $\mu^{p,z',w'}$ as an equilibrium measure. We construct this process as a limit of the up-down
Markov chains on the Gelfand-Tsetlin graph.

Denote by $X_{p,z',w'}(t)$ the image of the up Markov process
$H(t)$, corresponding to the coherent system $\{M_N^{p,0,z',w'}\}$,
under the map $\lambda\mapsto X(\lambda)$.

Denote by $U^N_{p,z',w'}(t)$ the image of the up-down Markov process
$T_N(t)$, corresponding to the coherent system
$\{M_N^{p,0,z',w'}\}$, under the map $\lambda\mapsto X(\lambda)$.

Finally, set
$$J^N_{p,z',w'}(t)=\pi_N\left(U^N_{p,z',w'}\left(\left\lfloor t\cdot N^2 \right\rfloor\right)\right)$$

Below we prove the following theorem:
\begin{theorem}
 \label{Theorem_convergence_process}
 There exists a limit stationary Markov process $J_{p,z',w'}(t)$ on $\mathcal X$.
 The finite-dimensional distributions of $J^N_{p,z',w'}(t)$ converge
 as $N\to\infty$ to the corresponding finite-dimensional
 distributions of $J_{p,z',w'}(t)$.
\end{theorem}

We will compute the transition probabilities of $J_{p,z',w'}(t)$ and
the generator of $J_{p,z',w'}(t)$.

\section{Properties of Markov chain $X_{p,z',w'}(t)$}
In this section we compute one-dimensional distributions and
transition probabilities of the Markov chain $X_{p,z',w'}(t)$. In what follows
we omit indices and write $X(t)$.

Recall that one-dimensional distributions of $X(N)$
coincide with $P_N^{p,z',w'}$.  $P_N^{p,z',w'}$ is a pushforward of the measure
$M_N^{p,0,z',w'}$ under the map $\lambda\mapsto X(\lambda)$, and
\begin{multline*}
M_N^{z,w,z',w'}(\lambda)=\left( S_N(z,w,z',w') \right) ^{-1}\cdot
\operatorname{Dim}^2(\lambda) \\ \times \prod\limits_{i=1}^N
\frac{1}{\Gamma(z-\lambda_i+i)\Gamma(w+N+1+\lambda_i-i)\Gamma(z'-\lambda_i+i)\Gamma(w'+N+1+\lambda_i-i).}
\end{multline*}
By Weyl's dimension formula (see e.g. \cite{Zh}):
 \begin{multline*} \operatorname{Dim}(\lambda)=
  \prod\limits_{1\le i<j \le
  N}\frac{\lambda_i-\lambda_j+j-i}{j-i}\\=
   \prod\limits_{1\le i<j \le N}\frac{
   (\lambda_i-i+N)-(\lambda_j-j+N)}{j-i},\quad
   \lambda\in\mathbb{GT}_N.
 \end{multline*}

\begin{lemma}
\label{lemma_vond_compl}
 Let $A=\{0,1,\dots,k\}$, $X\subset A$,
$\overline X = A\setminus X$. Denote
$$V(X)=\prod\limits^{x\in X,y\in X}_{x<y} (y-x).$$

Then we have
$$
 V(X)=V(\overline X)\prod\limits_{x\in \overline
 X}\frac{1}{x!(k-x)!}\cdot\prod\limits_{i=1}^k i!
$$
\end{lemma}
\begin{proof}
It is clear that
$$
 V(X)\prod\limits_{x\in\overline X}\left(\prod\limits_{i\neq x}^{i\in
 A}(x-i)\right)= V(A)\cdot V(\overline X).
$$
To complete the proof we observe that
$$
 V(A)=\prod\limits_{i=1}^k i!
$$
and
$$
 \prod\limits_{i\neq x}^{i\in
 A}(x-i)=
  x!(k-x)!.
$$

\end{proof}

As a corollary of Lemma \ref{lemma_vond_compl} and Weyl's dimension
formula we have

\begin{lemma}
\label{lemma_dim_formula} Let $\lambda\in\mathbb{GT}_N$ and
$X(\lambda)=X=(x_1<x_2<\dots<x_p),)$ then
$$
 \operatorname{Dim}(\lambda)=V(X(\lambda))\cdot
 \prod\limits_{i=1}^p\frac{1}{x_i!(N+p-1-x_i)!}\cdot
 \prod\limits_{i=1}^p(N+i-1)!.
$$
\end{lemma}

Now denote
\begin{equation}
 \label{weight_definition}
 w_N(x)=\frac{\Gamma(z'+N-x)\Gamma(w'+x+1)}{\Gamma(N+p-x)\Gamma(x+1)}.
\end{equation}

Below we use the Pochhammer symbol
$$
 (a)_k=\frac{\Gamma(a+k)}{\Gamma(a)}=a(a+1)\dots(a+k-1)
$$

\begin{proposition}
\label{proposition_one-time_distrib}
 For $X=(x_1<\dots<x_p)$ we have
$$P_N^{p,z',w'}(X)=Z^{p,z',w'}_N\cdot V^2(X)\prod\limits_{i=1}^p
w_N(x_i),$$ where $Z^{p,z',w'}_N$ is a normalization constant
$$
Z^{p,z',w'}_N=\prod\limits_{i=1}^N \frac
   {(i)_p}{(z'+w'+i)_p} \cdot
   \prod\limits_{i=1}^p \frac{1}{\Gamma(w'+i+1)\Gamma(z'-i+1))}
$$
\end{proposition}
\begin{proof}
 Straightforward computation using Lemmas \ref{lemma_vond_compl} and \ref{lemma_dim_formula}.
\end{proof}

Now suppose that $\lambda\in\mathbb{GT}_N$,
$\mu\in\mathbb{GT}_{N+1}$, $X=(x_1<\dots<x_p)=X(\lambda)$,
$X'=(x'_1<\dots<x'_p)=X(\mu)$. We write $X\prec X'$, provided that
 $\lambda\prec\mu$. Note that $X\prec X'$ if and only if for every $i$, $x'_i=x_i$
 or $x'_i=x_i+1$.

{\bf Remark.} Consider an arbitrary path
$\tau=(\tau(0),\tau(1),\dots)$ in $\mathbb{GT}$. Let
$X(N)=X(\tau(N))=(x_1^N<\dots<x_p^N)$. For every $1\le i\le p$ and
$N\ge 0$ we put a point on the plane $\mathbb Z^2$ with coordinates
$(N,x_i^N)$ and draw a segment connecting $(N,x_i^N)$ with
$(N+1,x_i^{N+1})$. In this way every path $\tau$ corresponds to a
collection of $p$ non-intersecting paths on the plane $\mathbb Z^2$.
This interpretation shows deep similarities between the objects of
the present paper and the objects of the paper \cite{Gor} where
collections of non-intersecting paths inside a hexagon were studied.

Our next aim is to obtain explicit formulas for the transition
probabilities of the process $X_{p,z',w'}(t)$.
\begin{proposition}
Suppose $X=(x_1<\dots<x_p)\in (X_{N,p})^p$ and
$X'=(x'_1<\dots<x'_p)\in (X_{N+1,p})^p$, then
\begin{multline*}
{\rm Prob}\{X(N+1)=X'\mid X(N)=X\}\\=\begin{cases}
\frac{1}{(z'+w'+N+1)_p}\frac{V(X')}{V(X)}\prod\limits_{i:x'_i=x_i}
(z'+N-x_i)\prod\limits_{i:x'_i=x_i+1}(w'+1+x_i),\quad X\prec X'\\
0,\quad otherwise.
\end{cases}
\end{multline*}
\end{proposition}
\begin{proof}
 By the definition ${\rm Prob}\{X(N+1)=X'\mid X(N)=X\}$ is transition
 function $p^\uparrow(\mu\mid \lambda)$, where $X'=X(\mu)$ and
 $X=X(\lambda)$. Lemma \ref{lemma_dim_formula} and Proposition~\ref{proposition_one-time_distrib} imply that
 \begin{multline*}
 {\rm Prob}\{X(N+1)=X'\mid X(N)=X\}=p^\uparrow(\mu\mid \lambda)=
 \frac{M_{N+1}(\mu)}{M_N(\lambda)}\cdot\frac{\operatorname{Dim}(\lambda)}{\operatorname{Dim}(\mu)}\\
 =\frac{Z^{p,z',w'}_{N+1}\cdot V^2(X')\prod\limits_{i=1}^p
w_{N+1}(x'_i)}{Z^{p,z',w'}_N\cdot V^2(X)\prod\limits_{i=1}^p
w_N(x_i)} \frac{V(X)\cdot
 \prod\limits_{i=1}^p\frac{1}{x_i!(N+p-1-x_i)!}\cdot
 \prod\limits_{i=1}^p(N+i-1)!}
 {V(X')\cdot
 \prod\limits_{i=1}^p\frac{1}{x'_i!(N+p-x'_i)!}\cdot
 \prod\limits_{i=1}^p(N+i)!}
 \end{multline*}
 Using \eqref{weight_definition} we obtain the desired formula.
\end{proof}

\section{Determinantal form of transition probabilities}
In this section we introduce determinantal formulas for the
transition probabilities of the Markov chains  under
consideration. These formulas are very important for our arguments.

Denote
$$
c_i^N=\sqrt{\left(1-\frac{i}{p+N}\right)\left(1+\frac{i}{w'+z'+N+1}\right)}.
$$
Let $v_N(x,y)$ be $(N+p)\times(N+p-1)$ matrix ($0\le x\le N+p-1$,
$0\le y \le N+p$) given by
$$
v_N(x,y)=\begin{cases}
 \sqrt{\frac{(w'+x+1)(x+1)}{(x+N)(w'+z'+N+1)}},& y=x+1\\
 \sqrt{\frac{(z'+N-x)(p+N-x)}{(p+x)(w'+z'+N+1)}},& y=x\\
 0,& otherwise.
\end{cases}
$$

\begin{proposition}
\begin{multline}
\label{tr_prob_eq} {\rm Prob}\{X(N+1)=X'\mid
X(N)=X\}\\=\frac{\sqrt{P_{N+1}^{p,z',w'}(X')}}{\sqrt{P_{N}^{p,z',w'}(X)}}\cdot
\det[v_N(x_i,x'_j)]_{i,j=1,\dots,p}\cdot\frac{1}{\prod_{i=0}^{p-1}c_i^N},
\end{multline}
where $[v_N(x_i,x'_j)]_{i,j=1,\dots,p}$ is a submatrix of the matrix
$v_N(x,y)$.
\end{proposition}
\begin{proof}
Observe that $v_N(x,y)$ is a two-diagonal matrix. Any submatrix of a
two-diagonal matrix, which has non-zero determinant, is
block-diagonal, where each block is either upper or lower triangular
matrix. Thus, any non-zero minor is a product of suitable matrix
elements. Consequently, if $ X\prec X'$, then
\begin{multline*}
\frac{\sqrt{P_{N+1}^{p,z',w'}(X')}}{\sqrt{P_{N}^{p,z',w'}(X)}}\cdot
\det[v_N(x_i,x'_j)]_{i,j=1,\dots,p}\cdot\frac{1}{\prod_{i=0}^{p-1}c_i^N}
\\
=\sqrt{\frac{Z^{p,z',w'}_{N+1}}{Z^{p,z',w'}_N}}\cdot\frac{1}{\prod_{i=0}^{p-1}c_i^N}\cdot
\frac{V(X')}{V(X)}\cdot
\prod\limits_{i=1}^p\sqrt{\frac{w_{N+1}(x_i')}{w_N(x_i)}}
\\ \times
\prod\limits_{i:x'_i=x_i}
\sqrt\frac{(z'+N-x_i)(p+N-x_i)}{(p+x_i)(w'+z'+N+1)}\prod\limits_{i:x'_i=x_i+1}\sqrt\frac{(w'+x+1)(x+1)}{(x+N)(w'+z'+N+1)}
\\=
\frac{\Gamma(z'+w'+N+1)}{\Gamma(z'+w'+p+N+1)}\frac{V(X')}{V(X)}\prod\limits_{i:x'_i=x_i}
(z'+N-x_i)\prod\limits_{i:x'_i=x_i+1}(w'+1+x_i)
\\
={\rm Prob}\{X(N+1)=X'\mid X(N)=X\},
\end{multline*}

and both sides of \eqref{tr_prob_eq} are equal to zero if  $X\nprec
X'$.
\end{proof}

Let $Q^k_{w',z'-p,N+p-1}(x)$ be the Hahn polynomial of degree $k$.
These polynomials are orthogonal with respect to the weight $w_N(x)$
(see \cite{KS}).

Denote
$$f^k_N(x)=\frac{Q^k_{w',z'-p,N+p-1}(x)\sqrt{w_N(x)}.}{\sqrt{(Q^k_{w',z'-p,N+p-1},Q^k_{w',z'-p,N+p-1})}},$$
where $(Q^k_{w',z'-p,N+p-1},Q^k_{w',z'-p,N+p-1})$ is the squared
norm of $Q^k_{w',z'-p,N+p-1}$ in $L_2(\{0,1,\dots,N+p-1\},w_N(x))$. We provide the exact value of this norm in Appendix.

Functions $f^k_N,\quad {k=0,1,\dots,N+p-1}$ form an orthonormal
bases in $L_2(\{0,1,\dots ,N+p-1\})$ (this $L_2$ is with respect to
the uniform measure).

{\bf Remark.} Here and below we use definitions of classical
orthogonal polynomials given in \cite{KS}. Thus
$$
Q^k_{x,\alpha,\beta,M}=
\mathstrut_3F_2{{-k,-x,k+\alpha+\beta+1}\choose {-M,\alpha+1}}\bigl(1\bigr).
$$

We also use the book \cite{NSU}, which uses slightly different
definition of these polynomials (two definitions differ by a
factor). All formulas in the present paper are written according to
the definitions of \cite{KS}.

\begin{proposition} We have
\label{Proposition_Hahn_eigen_relation}
$$
 v_N(x,y)=\sum_{i=0}^{N+p-1} c_i^N f^i_N(x)f^i_{N+1}(y).
$$
Equivalently,
$$v_N=(F_N)^T\cdot C_N\cdot (F_{N+1}),$$
where $v_N$ is $(N+p)\times (N+p+1)$ matrix $v_N(x,y)$,  $F_m$ is
$(m+p)\times (m+p)$ matrix $f_m^i(j)_{i,j=0,\dots,m+p-1}$ and $C_N$
is $(N+p)\times (N+p+1)$ matrix given by
$$
(C_N)_{ij}=\begin{cases} c_i^N,& \text{if } i=j,\\
            0,&otherwise.
            \end{cases}
$$
\end{proposition}
\begin{proof}
The argument repeats \cite[Lemma 8]{Gor}. More details are given in
Appendix.
\end{proof}

Our next aim is to obtain a determinantal representation for the transition probabilities of the Markov chain
$U^N_{p,z',w'}(t)$.
\begin{lemma} If $X'\prec X$ then
\begin{multline*}
{\rm Prob}\{X(N)=X'\mid
X(N+1)=X\}\\=\frac{\sqrt{P_N^{p,z',w'}(X')}}{\sqrt{P_{N+1}^{p,z',w'}(X)}}\cdot
\det[v_N(x'_i,x_j)]_{i,j=1,\dots,p}\cdot\frac{1}{\prod_{i=0}^{p-1}c_i^N},
\end{multline*}
\end{lemma}
\begin{proof}
\begin{multline*}
{\rm Prob}\{X(N)=X'\mid  X(N+1)=X\}
\\=
{\rm Prob}\{X(N+1)=X\mid X(N+1)=X'\}\frac{{\rm Prob}\{X(N)=X'\}}{{\rm
Prob}\{X(N+1)=X\}}
\\=
\frac{\sqrt{P_{N+1}^{p,z',w'}(X)}}{\sqrt{P_{N}^{p,z',w'}(X')}}\cdot
\det[v_N(x'_i,x_j)]_{i,j=1,\dots,p}\cdot\frac{1}{\prod_{i=0}^{p-1}c_i^N}\cdot\frac{P_N^{p,z',w'}(X')}{P_{N+1}^{p,z',w'}(X)}
\\=
\frac{\sqrt{P_N^{p,z',w'}(X')}}{\sqrt{P_{N+1}^{p,z',w'}(X)}}\cdot
\det[v_N(x'_i,x_j)]_{i,j=1,\dots,p}\cdot\frac{1}{\prod_{i=0}^{p-1}c_i^N}
\end{multline*}
\end{proof}
\begin{proposition}
\label{prop_one-step_tr_prob_U} Transition probabilities of the
process $U^N_{p,z',w'}(t)$ are given by

\begin{multline*}
 {\rm Prob}\{U^N_{p,z',w'}(t+1)=X'\mid
U^N_{p,z',w'}(t)=X\}\\=
\frac{\sqrt{P_{N}^{p,z',w'}(X')}}{\sqrt{P_{N}^{p,z',w'}(X)}}
\det[u_N(x_i,x'_j)]_{i,j=1,\dots,p}\cdot\frac{1}{\prod_{i=0}^{p-1}(c_i^N)^2},
\end{multline*}
where $[u_N(x_i,x'_j)]_{i,j=1,\dots,p}$ is a submatrix of the $(N+p)\times
(N+p)$ matrix
$$u_N=v_N\cdot (v_N)^T$$
\end{proposition}
\begin{proof}
By the definition of $U^N_{p,z',w'}$ we have
\begin{multline*}
{\rm Prob}\{U^N_{p,z',w'}(t+1)=X'\mid  U^N_{p,z',w'}(t)=X\}\\
=\sum\limits_Y {\rm Prob}\{X(N)=X'\mid  X(N+1)=Y\}\cdot {\rm
Prob}\{X(N+1)=Y\mid X(N)=X\}.
\end{multline*}
Thus,
\begin{multline*}
 {\rm Prob}\{U^N_{p,z',w'}(t+1)=X'\mid  U^N_{p,z',w'}(t)=X\}\\=
 \frac{\sqrt{P_{N}^{p,z',w'}(X')}}{\sqrt{P_{N}^{p,z',w'}(X)}}\sum\limits_{y_1<y_2<\dots<y_p}
 \frac{\det[v_N(x_i,y_j)]_{i,j=1,\dots,p}\det[v_N(x'_i,y_j)]_{i,j=1,\dots,p}}{\prod_{i=0}^{p-1}(c_i^N)^2}.
\end{multline*}
Let $A$ be $p\times (N+p+1)$ matrix given by
$$
 A_{ij}=v_N(x_i,j-1),\quad i=1,\dots,p,\quad j=1,\dots, N+p+1.
$$
Let $B$ be $(N+p+1)\times p$ matrix given by
$$
 B_{ij}=v_N(i-1,x'_j),\quad i=1,\dots,N+p+1,\quad j=1,\dots, p.
$$

Denote by $A^{j_1j_2\dots j_p}$ the square submatrix of the matrix $A$ consisting of the columns $j_1,j_2,\dots,j_p$;
denote by $B^{j_1j_2\dots j_p}$ the square submatrix of the matrix $B$ consisting of the rows $j_1,j_2,\dots,j_p$.
By Cauchy-Binet identity
$$
 \det(AB)=\sum_{1\le j_1<j_2<\dots<j_p\le N+p+1} \det(A^{j_1j_2\dots j_p}\cdot B^{j_1j_2\dots j_p}).
$$
Observe that $(AB)_{ij}=u_N(x_i,x'_j)$. Consequently,
\begin{multline*}
{\rm Prob}\{U^N_{p,z',w'}(t+1)=X'\mid  U^N_{p,z',w'}(t)=X\}\\=
\frac{\sqrt{P_{N}^{p,z',w'}(X')}}{\sqrt{P_{N}^{p,z',w'}(X)}}
\frac{\det[u_N(x_i,x'_j)]_{i,j=1,\dots,p}}{\prod_{i=0}^{p-1}(c_i^N)^2},
\end{multline*}
\end{proof}
\begin{proposition}
\label{proposition_U_tr_prob_long} Let $k\in\mathbb Z^+$, then
\begin{multline}
\label{tr_prob_U} {\rm Prob}\{U^N_{p,z',w'}(t+k)=X'\mid
U^N_{p,z',w'}(t)=X\}\\=
\frac{\sqrt{P_{N}^{p,z',w'}(X')}}{\sqrt{P_{N}^{p,z',w'}(X)}}
\frac{\det[w_{N,k}(x_i,x'_j)]_{i,j=1,\dots,p}}{\prod_{i=0}^{p-1}(c_i^N)^{2k}},
\end{multline}
where
$$
 w_{N,k}(x,x')=\sum_{i=0}^{N+p-1} (c_i^N)^{2k} f^i_N(x)f^i_{N}(x').
$$
\end{proposition}
\begin{proof}
We have
\begin{multline*}
{\rm Prob}\{U^N_{p,z',w'}(t+k)=X'\mid  U^N_{p,z',w'}(t)=X\}\\
=\sum\limits_Y {\rm Prob}\{U^N_{p,z',w'}(t+k)=X'\mid
U^N_{p,z',w'}(t+1)=Y\}\\ \cdot
{\rm Prob}\{U^N_{p,z',w'}(t+1)=Y\mid  U^N_{p,z',w'}(t)=X\}\\
=\sum\limits_Y {\rm Prob}\{U^N_{p,z',w'}(t+k-1)=X'\mid
U^N_{p,z',w'}(t)=Y\}\\ \cdot {\rm Prob}\{U^N_{p,z',w'}(t+1)=Y\mid
U^N_{p,z',w'}(t)=X\}.
\end{multline*}
Applying Proposition \ref{prop_one-step_tr_prob_U} and Cauchy-Binet
identity by induction we obtain the formula \eqref{tr_prob_U}, where
$[w_{N,k}(x_i,x'_j)]_{i,j=1,\dots,p}$ is a submatrix of the matrix
$w_{N,k}=\left(u_N\right)^k$. Consequently,
$$
w_{N,k}=(v_N (v_N)^T)^k=( (F_N)^T\cdot C_N\cdot F_{N+1}\cdot (F_{N+1})^T\cdot (C_N)^T\cdot F_{N})^k.
$$
Recall that
$$
 (F_m)_{ij}=f_m^i(j),\quad i,j=0,\dots,m+p-1.
$$
Since functions $f_m^i$ form an orthonormal basis,
$F_m\cdot (F_m)^T=Id$,
where $Id$ means the identity matrix.
Thus,
$$
w_{N,k}=(F_N)^T\cdot (C_N(C_N)^T)^k\cdot F_N.
$$
This equality implies Proposition
\ref{proposition_U_tr_prob_long}.
\end{proof}

\section{Limit process}
Recall that we are interested in studying
the stationary process
$$J_{p,z',w'}(t)=\lim_{N\to\infty}J^N_{p,z',w'}(t),$$
 where
$$J^N_{p,z',w'}(t)=\pi_N\left(U^N_{p,z',w'}\left(\left\lfloor t\cdot N^2 \right\rfloor\right)\right)$$
 and
$$\pi_N:(x_1,x_2,\dots,x_p)\mapsto \left(\frac{x_1}{N+p-1},\frac{x_2}{N+p-1},\dots,\frac{x_p}{N+p-1} \right).$$

Let us introduce some notation. Denote by $w_{p,z',w'}(x)$ the density function of $B$--distribution:
$$
w_{p,z',w'}(x)=x^{w'}(1-x)^{z'-p},\quad 0<x<1.
$$
Let $Jac^k_{z'-p,w'}(x)$ be \emph{the orthogonal Jacobi polynomial}
of the power $k$. These polynomials are defined for $x\in(0,1)$ and
are orthogonal with respect to the weight function $w_{p,z',w'}$.
According to \cite{KS}
$$
 Jac^n_{\alpha,\beta}(x)=\frac{(\alpha+1)_n}{n!} \mathstrut_2F_1{ {-n, n+\alpha+\beta+1}\choose {\alpha+1}}\bigl(1-x\bigr).
$$
Note that
while Jacobi polynomials are often defined for $x\in(-1,1)$, in the
present paper we scale and shift the domain of definition, and our
polynomials live on $(0,1)$.

Denote
$$
 j_{p,z',w'}^k(x)=\frac{Jac^k_{z'-p,w'}(x)\sqrt{w_{p,z',w'}(x)}}{\sqrt{(Jac^k_{z'-p,w'},Jac^k_{z'-p,w'})}},
$$
where $(Jac^k_{z'-p,w'},Jac^k_{z'-p,w'})$ is the squared norm of
$Jac^k_{z'-p,w'}$ in $L_2([0,1],w_{p,z',w'})$. We provide the exact value of this norm in Appendix.

It is clear that functions $j_{p,z',w'}^k,\quad k=0,1,\dots$ form
an orthonormal system in $L_2([0,1])$ with Lebesgue measure.

Finally, set
$$
K_{p,z',w'}(i)=i(i+w'+z'+1-p).
$$

In this section we prove the existence of the limit process $J_{p,z',w'}(t)$ (Theorem \ref{Theorem_convergence_process})
and compute its one-dimensional distributions and transition probabilities. Our main result is the following.

\begin{theorem}
\label{theorem_main_result}
 Finite-dimensional distributions of $J^N_{p,z',w'}(t)$ weakly converge
 as $N\to\infty$ to the corresponding finite-dimensional
 distributions of a limit process $J_{p,z',w'}(t)$.
$J_{p,z',w'}(t)$ is a stationary Markov process. Its initial
distribution is given by the density
$$
\rho_{p,z',w'}(X)= B_{p,z',w'}\prod\limits_{i<j}(x_i-x_j)^2\prod\limits_{i=1}^p
w_{p,z',w'}(x_i).
$$

The transition probabilities are given by the density
\begin{multline*}
{\mathcal P}^t_{p,z',w'}(Y\mid  X)
=\frac{\sqrt{\rho_{p,z',w'}(Y)}}{\sqrt{\rho_{p,z',w'}(X)}} \cdot
e^{tK_{p,z',w'}}\cdot\det[{\mathcal
J}_{p,z',w'}^t(x_i,y_j)]_{i,j=1,2,\dots,p},
\end{multline*}
where
$$
{\mathcal J}_{p,z',w'}^t(x,y)=\sum\limits_{i=0}^\infty
e^{-tK_{p,z',w'}(i)} {j}_{p,z',w'}^i(x) {j}_{p,z',w'}^i(y)
$$
and
$$
  K_{p,z',w'}=\sum_{i=0}^{p-1}K_{p,z',w'}(i).
$$
\end{theorem}
{\bf Remark.} Determinantal form of transition probabilities implies
that $J_{p,z',w'}(t)$ is space-time (dynamical) determinantal
process. See Proposition \ref{proposition_determinantal_kernel} for the exact
statement.

Let us begin the proof.

In the first place we want to recompute the density of the
one-dimensional distribution of $J_{p,z',w'}(t)$ (it was originally
computed in \cite{BO}).

\begin{proposition} We have
\label{proposition_density_lim}
\begin{multline*}
{\rm Prob}\{J_{p,z',w'}(t)\in dX\}=\rho_{p,z',w'}(X)dx_1\dots dx_p\\
=B_{p,z',w'}\prod\limits_{i<j}(x_i-x_j)^2\prod\limits_{i=1}^p
x_i^{w'} (1-x_i)^{z'-p} dx_i,
\end{multline*}
where $X=(x_1<x_2<\dots<x_p)$
\end{proposition}
\begin{proof}

Clearly, one-dimensional distributions of $J^N_{p,z',w'}(t)$ are
$$
{\rm Prob}\{J^N_{p,z',w'}(t)=X\}=P_N^{p,z',w'}( (N+p-1)\cdot X),
$$
where $X=(x_1<x_2<\dots<x_p)$ and $(N+p-1)\cdot X=(
(N+p-1)x_1<(N+p-1)x_2<\cdot<(N+p-1)x_p)\in \{0,1,\dots,N+p-1\}^p$.

If the limit below exist, then the following equality holds
\begin{equation}
\label{Convergence_of_measures} {\rm Prob}\{J_{p,z',w'}(t)\in
dX\}=\lim_{N\to\infty}(N+p-1)^p\cdot {\rm Prob}\{J^N_{p,z',w'}(t)=X\}.
\end{equation}
Applying the formula
$$
 \frac{\Gamma(M+a)}{\Gamma(M+b)}\sim M^{a-b},\quad M\to\infty,
$$
 (here $A\sim B$ means $\lim \frac{A}{B}=1)$
 we get:
\begin{multline*}
{\rm Prob}\{J^N_{p,z',w'}(t)=X\}=Z^{p,z',w'}_N\cdot
\prod\limits_{i<j}((N+p-1)x_i-(N+p-1)x_j)^2 \\ \cdot
\prod\limits_{i=1}^p
\frac{\Gamma(z'+N-(N+p-1)x_i)\Gamma(w'+(N+p-1)x_i+1)}{\Gamma(N+p-(N+p-1)x_i)\Gamma((N+p-1)x_i+1)}\\
\sim Z^{p,z',w'}_N\cdot \prod\limits_{i<j}(x_i-x_j)^2\cdot (N+p-1)^{p(p-1)} \\ \cdot
\prod\limits_{i=1}^p ( (N+p-1)-(N+p-1)x_i )^{z'-p} ( (N+p-1)x_i)^{w'}\\=
Z^{p,z',w'}_N\cdot (N+p-1)^{p(z'+w')-p)}\cdot\prod\limits_{i<j}(x_i-x_j)^2 \cdot \prod\limits_{i=1}^p ( 1-x_i )^{z'-p} x_i^{w'}
\end{multline*}
Substituting $Z^{p,z',w'}_N$ from Proposition
\ref{proposition_one-time_distrib} we obtain
$$
{\rm
Prob}\{J^N_{p,z',w'}(t)=X\}=\frac{B_{p,z',w'}}{(N+p-1)^p}\prod\limits_{i<j}(x_i-x_j)^2
\cdot \prod\limits_{i=1}^p ( 1-x_i )^{z'-p} x_i^{w'}.
$$

\end{proof}
{\bf Remark.} Note that the convergence in
\eqref{Convergence_of_measures} is uniform on any compact set
$$D\subset \{(x_1,x_2,\dots,x_p): 0<x_1\le x_2 \le \dots\le x_p< 1 \}.$$

Next we concentrate on the multidimensional distributions of $J_{p,z',w'}(t)$.

\begin{proposition}
\label{proposition_multy_distributions_convergence} Consider $n$
distinct times
$$
 0\le t_1<t_2<\dots<t_n.
$$
As $N$ tends to infinity, $n$--dimensional distribution of $J^N_{p,z',w'}(t)$ corresponding to $t=t_1,\dots,t_n$ converges.
The limit distribution is given by its density
\begin{multline}
\label{limit_multy_distribution}
\rho(t_1,0< x_1^1<\dots<x_p^1< 1; \dots ; t_n,0< x_1^n<\dots<x_p^n< 1)
\\=
{\sqrt{\rho_{p,z',w'}(x_1^1,\dots,x_p^1)}}
\\ \times
\frac{\det[{\mathcal
J}_{p,z',w'}^{t_2-t_1}(x_i^1,x_j^2)]_{i,j=1,2,\dots,p}}{e^{-(t_2-t_1)K_{p,z',w'}}}
\dots
\frac{\det[{\mathcal
J}_{p,z',w'}^{t_n-t_{n-1}}(x_i^{n-1},x_j^n)]_{i,j=1,2,\dots,p}}{e^{-(t_n-t_{n-1})K_{p,z',w'}}}
\\ \times
{\sqrt{\rho_{p,z',w'}(x_1^n,\dots,x_p^n)}},
\end{multline}

where
${\mathcal J}_{p,z',w'}^t(x,y)$ and $K_{p,z',w'}$ are the same as in Theorem \ref{theorem_main_result}.
\end{proposition}

\begin{proof}
Recall that $J^N_{p,z',w'}(t)$ is a Markov process.
Proposition \ref{proposition_one-time_distrib} gives its one-dimensional distributions while
Proposition \ref{proposition_U_tr_prob_long} provides transition probabilities of $J^N_{p,z',w'}(t)$.

Consequently, the joint distribution of $J^N_{p,z',w'}(t_1),\dots,J^N_{p,z',w'}(t_n)$ is given by

\begin{multline*}
{\rm Prob}\{J^N_{p,z',w'}(t_1)=X^1, \dots
J^N_{p,z',w'}(t_n)=X^n\}\\=
{\sqrt{P_{N}^{p,z',w'}((N+p-1)X^1)}}
\\ \times  \frac{\det[w_{N,\lfloor
t_2\cdot N^2 \rfloor-\lfloor t_1\cdot N^2 \rfloor
}((N+p-1)x^1_i,(N+p-1)x^2_j)]_{i,j=1,\dots,p}}{\prod_{i=0}^{p-1}(c_i^N)^{2\lfloor
t_2\cdot N^2 \rfloor-2\lfloor t_1\cdot N^2 \rfloor}}
\times \dots\\
\times \frac{\det[w_{N,\lfloor t_n\cdot N^2 \rfloor -\lfloor
t_{n-1}\cdot N^2 \rfloor
}((N+p-1)x^{n-1}_i,(N+p-1)x^n_j)]_{i,j=1,\dots,p}}{\prod_{i=0}^{p-1}(c_i^N)^{2\lfloor
t_{n}\cdot N^2 \rfloor-2\lfloor t_{n-1}\cdot N^2 \rfloor}} \\
\times {\sqrt{P_{N}^{p,z',w'}((N+p-1)X^n)}}
\end{multline*}

Proposition \ref{proposition_density_lim} implies that as
$N\to\infty$,
$${\sqrt{P_{N}^{p,z',w'}((N+p-1)X)\cdot (N+p-1)^p}}\to {\sqrt{\rho_{p,z',w'}(X)}}.$$

Next observe that
\begin{multline*}
\lim_{N\to\infty}(c_i^N)^{2\lfloor t\cdot N^2 \rfloor} =
\lim_{N\to\infty}\left(\left(1-\frac{i}{p+N}\right)\left(1+\frac{i}{w'+z'+N+1}\right)\right)^{\lfloor
t\cdot N^2 \rfloor}
\\=
\lim_{N\to\infty}\left(1+\frac{i(p-w'-z'-1)-i^2}{(p+N)(w'+z'+N+1)}
\right)^{\lfloor t\cdot N^2 \rfloor}=e^{-i(i+w'+z'+1-p)t}
\end{multline*}
It remains to prove that for $t>s$, as $N\to\infty$
$$
(N+p-1)w_{N,\lfloor t\cdot N^2 \rfloor-\lfloor s\cdot N^2 \rfloor} ((N+p-1)x,(N+p-1)y) \to
{\mathcal J}_{p,z',w'}^{t-s}(x,y).
$$
In the following computations we omit parameters and write $Q^i$
instead of $Q^i_{w',z'-p,N+p-1}$. Also we write $\tilde x$ instead
of $(N+p-1)x$ and $\tilde y$ instead of $(N+p-1)y$. We have
\begin{multline*}
w_{N,\lfloor t\cdot N^2 \rfloor-\lfloor s\cdot N^2 \rfloor}(\tilde x,\tilde y)
=\sum_{i=0}^{N+p-1} (c_i^N)^{2\lfloor t\cdot N^2 \rfloor-2\lfloor s\cdot N^2 \rfloor}
{f^i_N(\tilde x)f^i_{N}(\tilde y)} \\
= \sqrt{w_N(\tilde x)w_N(\tilde
y)}  \cdot \sum_{i=0}^{N+p-1} (c_i^N)^{2\lfloor t\cdot N^2
\rfloor-2\lfloor s\cdot N^2 \rfloor}\frac{Q^i(\tilde x)Q^i(\tilde x)}{(Q^i,Q^i)}
\\\sim
\sqrt{w_{p,z',w'}(x)w_{p,z',w'}(y)}(N+p-1)
^{w'+z'-p}  \cdot \sum_{i=0}^{N+p-1} (c_i^N)^{2\lfloor t\cdot N^2
\rfloor-2\lfloor s\cdot N^2 \rfloor}\frac{Q^i(\tilde x)Q^i(\tilde x)}{(Q^i,Q^i)}
\\=
\sqrt{w_{p,z',w'}(x)w_{p,z',w'}(y)}(N+p-1)^{w'+z'-p}(A+B),
\end{multline*}

where
$$
A=\sum_{i=0}^{l} (c_i^N)^{2\lfloor t\cdot N^2
\rfloor-2\lfloor s\cdot N^2 \rfloor}\frac{Q^i(\tilde x)Q^i(\tilde x)}{(Q^i,Q^i)}
$$

and
$$
B=\sum_{i=l+1}^{N+p-1} (c_i^N)^{2\lfloor t\cdot N^2
\rfloor-2\lfloor s\cdot N^2 \rfloor}\frac{Q^i(\tilde x)Q^i(\tilde x)}{(Q^i,Q^i)}.
$$

Now fix large enough $l$ and send $N$ to $\infty$. First, let us
examine $A$. We have (see \cite[Section 2.5]{KS})
$$
Q^i_{w',z'-p,N+p-1}((N+p-1)x)\to
\frac{Jac^i_{w',z'-p}(1-x)}{Jac^i_{w',z'-p}(0)}=\frac{Jac^i_{z'-p,w'}(x)}{Jac^i_{z'-p,w'}(1)},
$$
and the convergence is uniform in $x$ belonging to any compact
subset of $(0,1)$. Here we used the relation
$$
 Jac^i_{\alpha,\beta}(1-x)=(-1)^i Jac^i_{\beta,\alpha}(x).
$$

By straightforward computation one proves that
$$
\lim_{N\to\infty}
\frac{(N+p-1)^{w'+z'+1-p}}{(Q^i_{w',z'-p,N+p-1},Q^i_{w',z'-p,N+p-1})}=\frac{(Jac^i_{z'-p,w'}(1))^2}{(Jac^i_{z'-p,w'}(x),Jac^i_{z'-p,w'}(x))},
$$
in slightly different form this claim was also proved in \cite{NSU}.

Consequently,
\begin{multline*}
 \lim_{N\to\infty}(N+p-1)\cdot\sqrt{w_{p,z',w'}(x)w_{p,z',w'}(y)}(N+p-1)^{w'+z'-p}A\\
 =
 \sum\limits_{i=0}^l
e^{-(t-s)K_{p,z',w'}(i)} {j}_{p,z',w'}^i(x) {j}_{p,z',w'}^i(y)\\
=
 {\mathcal J}_{p,z',w'}^t(x,y)- \sum\limits_{i=l+1}^\infty
e^{-(t-s)K_{p,z',w'}(i)} {j}_{p,z',w'}^i(x) {j}_{p,z',w'}^i(y)
\end{multline*}

Observe that for an arbitrary $\varepsilon>0$ we have
$$ \left|\sum\limits_{i=l+1}^\infty
e^{-(t-s)K_{p,z',w'}(i)} {j}_{p,z',w'}^i(x)
{j}_{p,z',w'}^i(y)\right|<\varepsilon,
$$
if $l$ is large enough. Indeed,
$$e^{-(t-s)K_{p,z',w'}(i)}<e^{-(t-s)i(i-1)}$$
and
$$|{j}_{p,z',w'}^i(x)|<e^{c(x)i},$$
where $c(x)$ is a bounded function for $x$ belonging to any compact
subset of $(0,1)$. This estimate follows, for instance, from the
recurrence relations on Jacobi polynomials.

Thus, to finish the proof we should show that
$$
\limsup_{N\to\infty}\sqrt{w_{p,z',w'}(x)w_{p,z',w'}(y)}(N+p-1)^{w'+z'+1-p}|B|<\varepsilon(l),
$$
and $\varepsilon(l)$ tends to zero when $l$ tends to $\infty$.

We need the following lemma:

\begin{lemma} For large enough $N$ and any $i<N+p$ we have
 $$(c_i^N)^{2\lfloor t\cdot N^2\rfloor-2\lfloor s\cdot N^2 \rfloor}<e^{-c\cdot i  \ln^2(i) },$$
 where constant $c$ does not depend on $N$.
\end{lemma}

\begin{proof}
 $$
  (c_i^N)^2=1-\frac{i(i+w'+z'-p+1)}{(p+N)(w'+z'+N+1)}<1-c_1\frac{i^2}{N^2}.
 $$
 If $i<\frac{N}{\ln(N)}$ and $N$ is large enough, then
 \begin{multline*}
 \left(1-c_1\frac{i^2}{N^2}\right)^{\lfloor t\cdot N^2\rfloor-\lfloor s\cdot N^2 \rfloor}<
 \left(1-c_1\frac{i^2}{N^2}\right)^{ c_2\cdot N^2}=
 \left(\left(1-c_1\frac{i^2}{N^2}\right)^{\frac{N^2}{c_1 i^2}}\right)^{c_1c_2
 i^2}\\<(1-\varepsilon)^{c_1c_2 i^2}<e^{c_3 i \ln^2 i}.
 \end{multline*}

 If $i\ge \frac{N}{\ln(N)}$ then
\begin{multline*}
 \left(1-c_1\frac{i^2}{N^2}\right)^{\lfloor t\cdot N^2\rfloor-\lfloor s\cdot N^2 \rfloor}<
 \left(1-c_1\frac{i^2}{N^2}\right)^{ c_2\cdot N^2} \le
 \left(1-c_1\frac{ (N/\ln(N))^2}{N^2}\right)^{ c_2\cdot N^2}\\=
 \left(\left(1-\frac{c_1}{\ln(N)}\right)^{\frac{\ln(N)}{c_1}}\right)^{c_1c_2\cdot \frac{N^2}{\ln(N)}}
 <
 (1-\varepsilon)^{c_1c_2\cdot\frac{N^2}{\ln(N)}}<(1-\varepsilon)^{c_1c_2
 \frac{i^2}{\ln(i)}}<e^{c_4\cdot i \ln^2 i}
 \end{multline*}

\end{proof}

Recurrence relations on Hahn polynomials (see e.g.
\cite[(1.5.3)]{KS} imply that
$$
 |Q^i_{w',z'-p,N+p-1}(x)|<e^{c'\cdot i \ln(i)},
$$
 where constant $c'$ does not depend on either $N$ or $x$.

Next note that
$$
 \frac{(N+p-1)^{w'+z'+1-p}}{(Q^i_{w',z'-p,N+p-1},Q^i_{w',z'-p,N+p-1})}
$$
is bounded from above by $e^{c''\cdot i \ln(i)}$.

For large enough $N$ and $l$ we obtain the following estimate
\begin{multline*}
(N+p-1)^{w'+z'+1-p}|B|= \\ \Biggl|\sum_{i=l+1}^{N+p-1}
(c_i^N)^{2\lfloor t\cdot N^2
\rfloor-2\lfloor s\cdot N^2 \rfloor}Q^i_{w',z'-p,N+p-1}((N+p-1)x)Q^i_{w',z'-p,N+p-1}((N+p-1)x)\\ \times
\frac{(N+p-1)^{w'+z'+1-p}}{(Q^i_{w',z'-p,N+p-1},Q^i_{w',z'-p,N+p-1})}\biggr)\Biggr|\\
<\sum_{i=l+1}^{N+p-1}e^{-c i\ln^2(i)+2c' i\ln(i) +c'' i\ln(i)}<
\sum_{i=l+1}^{N+p-1}e^{-\tilde c
i\ln^2(i)}<\sum_{i=l}^{\infty}e^{-\tilde c i}=\frac{e^{-\tilde c
l}}{1-e^{\tilde c l}}.
\end{multline*}
Observing that $\frac{e^{-\tilde c l}}{1-e^{-\tilde c l}}\to 0$ when
$l\to\infty$, completes the proof of Proposition
\ref{proposition_multy_distributions_convergence}.

\end{proof}

Since multidimensional distributions of
the process $J^N_{p,z',w'}(t)$ converge, we define
$J_{p,z',w'}(t)$ as a limit process. Formulas
\eqref{limit_multy_distribution} for the multidimensional
distributions imply that $J_{p,z',w'}(t)$ is a stationary Markov process
with initial distribution and transition probabilities as in Theorem
\ref{theorem_main_result}. Thus, Theorems \ref{Theorem_convergence_process} and \ref{theorem_main_result}
 are proved.

\smallskip

The Markov process $J_{p,z',w'}(t)$ has one interesting feature. Its
dynamical (space-time) \emph{correlation functions} can be expressed
as minors of a certain \emph{extended kernel}.

Let $\rho_n(x_1,t_1;x_2,t_2;\dots;x_n,t_n)$ be the $n$th correlation function
of $J_{p,z',w'}(t)$. Informally $\rho_n$ can be defined by
$$
 {\rm Prob}\{x_1\in J_{p,z',w'}(t_1),\dots, x_n \in
J_{p,z',w'}(t_n)\}=\rho_n(x_1,t_1;\dots;x_n,t_n)dx_1\dots dx_n.
$$

\begin{proposition}
 \label{proposition_determinantal_kernel}
Consider $n$ distinct points $(x_1,t_1),\dots,(x_n,t_n)$. We have
$$
\rho_n(x_1,t_1;x_2,t_2;\dots;x_n,t_n)= \det[{\rm Ker}_{p,z',w'}(x_i,t_i;x_j,t_j)]_{i,j=1,\dots,n}
,$$
where
$$
 {\rm Ker}_{p,z',w'}(x,t; y,s)=\begin{cases} \sum\limits_{i=0}^{p-1}
e^{(t-s)K_{p,z',w'}(i)} {j}_{p,z',w'}^i(x) {j}_{p,z',w'}^i(y),
&\text{if }
t\ge s,\\
-\sum\limits_{i=p}^\infty e^{(t-s)K_{p,z',w'}(i)} {j}_{p,z',w'}^i(x)
{j}_{p,z',w'}^i(y), &\text{if } t<s.
\end{cases}
$$
\end{proposition}

We are not going to give the proof of this claim here. It follows from Theorem \ref{theorem_main_result} and Eynard-Metha
theorem. See \cite{EM} and \cite[Section 7.4]{BO_Meix}.

\cite{S} is a good survey on determinantal point processes. Additional information about extended
kernels can be found in e.g. \cite{EM},\cite{J}, \cite{TW}.

\section{Markov semigroup of the limit process}

Let $\mathcal W_p$ be the \emph{Weyl chamber}.
$$
\mathcal W_p=\{(x_1,\dots,x_p)\in [0,1]^p:x_1\le x_2\le\dots\le x_p\}.
$$
Clearly, $\mathcal W_p$ is a state space of the process
$J_{p,z',w'}(t)$. Recall that $\mu_{p,z',w'}$ is the one-dimensional
distribution of $J_{p,z',w'}(t)$. According to Theorem
\ref{theorem_main_result}, $\mu_{p,z',w'}$ is a probability measure
on $\mathcal W_p$ given by its density $\rho_{p,z',w'}$.

Denote by $M_{p,z',w'}^t$ the Markov semigroup of the process $J_{p,z',w'}(t)$ on $L_2(\mathcal W_p,\mu_{p,z',w'})$.
$M_{p,z',w'}^t$ acts on a function $f(X)\in L_2(\mathcal W_p,\mu_{p,z',w'})$ as follows
$$
 (M_{p,z',w'}^tf)(X)=\int_{\mathcal W_p}F(Y)\tilde P^t_{p,z',w'}(Y\mid  X) dY,
$$
where the integration is performed with respect to the Lebesgue measure on $\mathcal W_p$.

In this section we are going to compute the eigenfunctions and the generator of $M_{p,z',w'}^t$.

Recall that $Jac^k_{z'-p,w'}(x)$ is Jacobi polynomial of degree $k$
on $(0,1)$. Now we want to introduce \emph{multi-dimensional Jacobi
polynomials}. Let $\lambda=(\lambda_1\ge\lambda_2\ge\dots\ge
\lambda_p\ge 0)$ be a partition of $n$. Denote
$$
 Jac^\lambda_{z-p',w'}(x_1,\dots,x_p)=\frac{\det[Jac^{\lambda_i+p-i}_{z'-p,w'}(x_j)]_{i,j=1,\dots,p}}{\prod\limits_{i>j}(x_i-x_j)}.
$$
 It is clear that $Jac^\lambda_{z'-p,w'}(x_1,\dots,x_p)$ is a symmetric polynomial in $p$ variables of degree $n=\lambda_1+\dots+\lambda_p$.

We may view $Jac^\lambda_{z'-p,w'}(x_1,\dots,x_p)$ as a function on
$\mathcal W_p$. In what follows we use the same notation for
these functions.

It is well-known that the polynomials $Jac^\lambda_{z'-p,w'}(x_1,\dots,x_p)$ form an orthogonal basis in $L_2(\mathcal W_p,\mu_{p,z',w'})$.

\begin{theorem}
\label{Theorem_Semigroup_eigenfunctions}
 Polynomials $Jac^\lambda_{z'-p,w'}(x_1,\dots,x_p)$ are eigenfunctions of the Markov semigroup $M_{p,z',w'}^t$.
 $$
  M_{p,z',w'}^t(Jac^\lambda_{z'-p,w'})=c(\lambda,t)Jac^\lambda_{p,z',w'},
 $$
 where
 $$
 c(\lambda,t)=e^{t\left(\sum\limits_{i=1}^p K_{p,z',w'}(p-i)-
 K_{p,z',w'}(\lambda_i+p-i)\right)}
 $$
\end{theorem}
\begin{proof}
 \begin{multline*}
 (M_{p,z',w'}^tJac^\lambda_{z'-p,w'})(x_1,x_2,\dots,x_p)
 \\=
  \int_{y_1<\dots< y_p}Jac^\lambda_{z'-p,w'}(y_1,\dots,y_p)
\frac{\sqrt{\rho_{p,z',w'}(y_1,\dots,y_p)}}{\sqrt{\rho_{p,z',w'}(x_1,\dots,x_p)}}
\\ \times e^{tK_{p,z',w'}}\cdot\det[{\mathcal
J}_{p,z',w'}^t(x_i,y_j)]_{i,j=1,2,\dots,p} dy_1\dots dy_p
\\=\frac{e^{tK_{p,z',w'}}}{\prod_{i>j}(x_i-x_j)}
\int_{y_1<\dots<y_p}Jac^\lambda_{z'-p,w'}(y_1,\dots,y_p)\cdot
\prod_{i>j}(y_i-y_j)
\\ \times
\frac{\prod_{i=1,p}\sqrt w_{p,z',w'}(y_i)}{\prod_{i=1,p}\sqrt
w_{p,z',w'}(x_i)}
 \cdot\det[{\mathcal J}_{p,z',w'}^t(x_i,y_j)]_{i,j=1,2,\dots,p} dy_1\dots
 dy_p
\end{multline*}

Note that the integrand is a symmetric function in $y_1,\dots,y_p$. Thus, we can
integrate over $p$--dimensional cube $(0,1)^p$, instead of the simplex ${0<y_1<\dots<y_p<1}$.
We arrive to the following expression

\begin{multline*}
   \frac{e^{tK_{p,z',w'}}}{p!\prod_{i>j}(x_i-x_j)}
\int_{0<y_i<1}
\det[Jac^{\lambda_i+p-i}_{z'-p,w'}(y_j)]_{i,j=1,\dots,p}
\\ \times
\det\left[\sum_{n=0}^{\infty}e^{-tK_{p,z',w'}(n)}\frac{w_{p,z',w'}(y_j)Jac^n_{z'-p,w'}(y_j)Jac^n_{z'-p,w'}(x_i)}{(Jac^n_{z'-p,w'},Jac^n_{z'-p,w'})}\right]_{i,j=1,\dots,p}
dy_1\dots dy_p.
 \end{multline*}
Multiplying two matrices under determinants and using orthogonality
relations for the Jacobi polynomials we obtain the desired formula.
\end{proof}

Let $G_{p,z',w'}$ be the infinitesimal generator of the Markov semigroup $M_{p,z',w'}^t$.
Theorem \ref{Theorem_Semigroup_eigenfunctions} implies the following proposition.

\begin{proposition}
\label{proposition_basic_eigenfunctions} Polynomials
$Jac^\lambda_{z-p',w'}(x_1,\dots,x_p)$ are eigenfunctions of
$G_{p,z',w'}$.
 $$
  G_{p,z',w'}(Jac^\lambda_{z'-p,w'})=\tilde c(\lambda)Jac^\lambda_{z'-p,w'},
 $$
 where
 $$
 \tilde c(\lambda)=\sum\limits_{i=1}^p K_{p,z',w'}(p-i)-
 K_{p,z',w'}(\lambda_i+p-i)
 $$
\end{proposition}

Our next aim is to obtain explicit formulas for the generators $G_{p,z',w'}$.

Observe that the numbers $-K_{p,z',w'}(i)=-i(i+w'+z'+1-p)$ are
eigenvalues of the \emph{Jacobi differential operator} corresponding
to the eigenvectors $Jac^k_{z'-p,w'}(x)$. Namely, denote
$$
 D^{Jac}_{z'-p,w'}=x(1-x)\frac{d^2}{dx^2}+(w'+1-(w'+z'-p+2)x)\frac{d}{dx}.
$$
The following proposition holds (see e.g. \cite[(1.8.5)]{KS}).
\begin{proposition} We have
$$
 D^{Jac}_{z'-p,w'}(Jac^i_{z'-p,w'})=-K_{p,z',w'}(i)Jac^i_{z'-p,w'}.
$$
\end{proposition}

Set
$$
 {\mathcal D}_{z',w',p}=\sum_{i=1}^{p}
 x_i(1-x_i)\frac{\partial^2}{\partial x_i^2}+(w'+1-(w'+z'-p+2)x_i)\frac{\partial }{\partial x_i}.
$$

Recall that if $X=(x_1,\dots,x_p)$ then $V(X)=\prod\limits_{i>j}(x_i-x_j)$.
\begin{theorem} We have
\label{theorem_generator}
 $$
  G_{p,z',w'}(f)=\frac{1}{V(X)}\cdot {\mathcal D}_{z',w',p}(V(X)\cdot f) +
  K_{p,z',w'}f,
 $$
 and the domain of definition of $ G_{p,z',w'}$ consists of all
 functions ${f\in L_2(\mathcal W_p,\mu_{p,z',w'})}$ such that
 $$\frac{1}{V(X)}\cdot {\mathcal D}_{z',w',p}(V(x)\cdot f) +
  K_{p,z',w'}f\in L_2(\mathcal W_p,\mu_{p,z',w'}).$$
\end{theorem}
{\bf Remark.} After some simplifications the generator $G_{p,z',w'}$ can be alternatively expressed as
$$
G_{p,z',w'}(f)={\mathcal D}_{z',w',p}(f)+\sum_{i=1}^p \left(x_i(1-x_i)\sum_{j\neq i}\frac{1}{x_i-x_j}\right)\frac{\partial f}{\partial x_i}
$$
\begin{proof}
 It is sufficient to verify that
 $$G_{p,z',w'}(Jac^\lambda_{z-p',w'}) = \frac{1}{V(X)}\cdot {\mathcal D}_{z',w',p}\bigl(V(x)\cdot Jac^\lambda_{z-p',w'}\bigr) +
  K_{p,z',w'}Jac^\lambda_{z-p',w'}.$$
 We have
 \begin{multline*}\frac{1}{V(X)}\cdot {\mathcal D}_{z',w',p}\bigl(V(x)\cdot Jac^\lambda_{z-p',w'}\bigr) +
  K_{p,z',w'}Jac^\lambda_{z-p',w'}
  \\=
  \frac{ {\mathcal D}_{z',w',p}\det[Jac^{\lambda_i+p-i}_{z'-p,w'}(x_j)]_{i,j=1,\dots,p}}{V(X)}+
  K_{p,z',w'}Jac^\lambda_{z-p',w'}.
  \end{multline*}
  Let us write the determinant as the alternating sum of products
  $$
   A_\sigma=\prod\limits_{j=1}^p
   Jac^{\lambda_j+p-j}_{z'-p,w'}(x_{\sigma(j)}).
  $$
  (Here $\sigma\in S_p$ is an arbitrary permutation.)

   Proposition \ref{proposition_basic_eigenfunctions} implies that
   $${\mathcal D}_{z',w',p} A_\sigma=\left(-\sum_{j=1}^p
  K_{p,z',w'}(\lambda_j+p-j)\right)A_\sigma.$$
  Consequently,
  \begin{multline*}
   \frac{ {\mathcal D}_{z',w',p}\det[Jac^{\lambda_i+p-i}_{z'-p,w'}(x_j)]_{i,j=1,\dots,p}}{V(X)}+
  K_{p,z',w'}Jac^\lambda_{z-p',w'}
  \\=
  \Biggl(K_{p,z',w'}- \biggl(\sum_{j=1}^p
  K_{p,z',w'}(\lambda_j+p-j)\biggr)\Biggr)Jac^\lambda_{z-p',w'}=G_{p,z',w'}(Jac^\lambda_{z-p',w'})
  \end{multline*}

\end{proof}

\section{Doob h-transform interpretation}

In this section we explain that the process $J_{p,z',w'}(t)$ can be
viewed as $p$ independent identically distributed diffusions
conditioned (in the sense of Doob) never to collide, i.e. it is
\emph{the Doob h-transform} of $p$ independent processes.

\begin{proposition}
\label{proposition_eigen_vondermont} We have
 $$ {\mathcal D}_{z',w',p}V(X)=-K_{p,z',w'}V(X).$$
\end{proposition}
\begin{proof} This claim follows from Proposition
\ref{proposition_basic_eigenfunctions} and Theorem
\ref{theorem_generator}, but we also present an independent proof.
\begin{lemma}
\label{lemma_harmonicity_Konig} For arbitrary $a,b,c\in \mathbb R$
set
$$ (G f)(X)=\sum_{i=1}^{p}(x^2+ax+b)\frac{\partial^2}{\partial
x_i^2}+\left(-\frac{2}{3}(p-2)x_i+c\right)\frac{\partial}{\partial
x_i}.
$$
Then $GV(X)=0$.
\end{lemma}
\begin{proof}
 See \cite[Lemma 3.1]{KO}. Note that originally there was a misprint in
 \cite{KO}, the minus sign was missing.
\end{proof}

Lemma \ref{lemma_harmonicity_Konig} implies that
 \begin{multline*}
  {\mathcal D}_{z',w',p}V(X)=\left(-\frac{2}{3}(p-2)-(w'+z'-p+2)\right) \sum_{i=1}^p
  x_i\frac{\partial}{\partial
  x_i}V(X)
  \\=-\frac{p(p-1)}{2}\left(-\frac{1}{3}(p-2)+w'+z'\right)V(X).
 \end{multline*}
 To finish the proof let us compute explicitly $K_{p,z',w'}$:
 \begin{multline*}
  K_{p,z',w'}=\sum_{i=0}^{p-1}K_{p,z',w'}(i)=\sum_{i=0}^{p-1}
  (i^2+i(w'+z'-p+1))
  \\=\frac{p(p-1)(2p-1)}{6}+\frac{p(p-1)}{2}(w'+z'-p+1)
  \\=
  \frac{p(p-1)}{2}\left( \frac{2p-1}{3}-p+1+w'+z'\right)=
  \frac{p(p-1)}{2}\left(-\frac{1}{3}(p-2)+w'+z'\right).
 \end{multline*}
\end{proof}

Now consider $p$ independent diffusions with generators
$D^{Jac}_{z'-p,w'}$. We call these diffusions Jacobi processes; they
can be viewed as solutions of certain stochastic differential
equations.

The infinitesimal generator of these $p$ diffusions is precisely
${\mathcal D}_{z',w',p}$. Proposition
\ref{proposition_eigen_vondermont} implies that we can construct new
Markov process, which is a variant of Doob $h$--transform of $p$
diffusions. Here function $h$ is $V(X)$. The infinitesimal generator
of Doob $h$ transform with $h=V(X)$ is precisely  $G_{p,z',w'}$. The
Doob $h$--transforms and its properties were studied in numerous
papers, see e.g. \cite{D}, \cite{Konig}.

Thus, our process $J_{p,z',w'}(t)$ can be interpreted as Doob
$h$--transform of $p$ independent diffusions. In many examples Doob
$h$--transform coincides with the conditional process, given that
there is no collision of the components. In our case, clearly, there
are no collisions, trajectories of $J_{p,z',w'}(t)$ are collections
of $p$ nonintersecting paths. However the question whether
$J_{p,z',w'}(t)$ is precisely $p$ independent diffusions conditioned
on never having a collision between any two of its components,
remains open.

\section{Appendix. Some formulas for Hahn and Jacobi polynomials}

The following formulas can be found in e.g. \cite{KS}.

\begin{lemma} For $\alpha>-1$, $\beta>-1$,
\label{Lemma_Hahn_norm}
\begin{multline*}
(Q^k_{\alpha,\beta,M},Q^l_{\alpha,\beta,M})=
\sum_{x=0}^{M}\frac{\Gamma(\alpha+x+1)\Gamma(\beta+N-x+1)}{\Gamma(x+1)\Gamma(N-x+1)}Q^k_{\alpha,\beta,M}(x)Q^l_{\alpha,\beta,M}(x)\\
=\frac{(-1)^k(k+\alpha+\beta+1)_{M+1}(\beta+1)_k k! \Gamma(\alpha+1) \Gamma(\beta+1)}{(2k+\alpha+\beta+1)(\alpha+1)_k (-M)_k M!}\delta_{kl}.
\end{multline*}
\end{lemma}

\begin{lemma} For $x,y\in\{0,1,\dots,M\}$,
\label{Lemma_Hahn_dual_orthogonality}
\begin{multline*}
 \sum_{k=0}^M\frac{(2k+\alpha+\beta+1)(\alpha+1)_{k}(-M)_{k}M!}
    {(-1)^k(k+\alpha+\beta+1)_{M+1}(\beta+1)_{k}k!}Q_k(x,\alpha,\beta,M)
        Q_k(y,\alpha,\beta,M)=\\
      =\frac{\delta_{xy}}{
       \dfrac{(\alpha+1)_{x}(\beta+1)_{M-x}}{x!(M-x)!}}.
\end{multline*}
\end{lemma}

The following Lemma was proved in \cite[Lemma 8]{Gor}

\begin{lemma}
\label{Lemma_Hahn_recurrence_relation}
\begin{multline*}
\quad xQ_k(x-1;\alpha,\beta,M-1)+(M-x)Q_k(x;\alpha,\beta,M-1)
    =MQ_k(x;\alpha,\beta,M).
\end{multline*}
\end{lemma}

\begin{proof}[Proof of Proposition \ref{Proposition_Hahn_eigen_relation}]
 Recall that our goal is to prove the following equality
\begin{equation}
\label{formula_Hahn}
 \sum_{i=0}^{N+p-1} c_i^N f^i_N(x)f^i_{N+1}(y)=
 \begin{cases}
 \sqrt{\frac{(w'+x+1)(x+1)}{(x+N)(w'+z'+N+1)}},& y=x+1\\
 \sqrt{\frac{(z'+N-x)(p+N-x)}{(p+x)(w'+z'+N+1)}},& y=x\\
 0,& otherwise.
\end{cases}
\end{equation}
First, we express left-hand side of \eqref{formula_Hahn} through the polynomials $Q^k_{\alpha,\beta,M}$.
By the definition
$$
 f^k_N(x)=\frac{Q^k_{w',z'-p,N+p-1}(x)\sqrt{w_N(x)}}{\sqrt{(Q^k_{w',z'-p,N+p-1},Q^k_{w',z'-p,N+p-1})}}.
$$

Lemma \ref{Lemma_Hahn_norm} provides the value of $(Q^k_{w',z'-p,N+p-1},Q^k_{w',z'-p,N+p-1})$. Next, we express
$Q^k_{w',z'-p,N+p}$ through $Q^k_{w',z'-p,N+p-1}$ using Lemma \ref{Lemma_Hahn_recurrence_relation}. Applying
Lemma \ref{Lemma_Hahn_dual_orthogonality} completes the proof.
\end{proof}

The following lemma in slightly different form can be found in \cite{KS}.

\begin{lemma} For $\alpha>-1$, $\beta>-1$,
\label{Lemma_Jacobi_norm}
$$
\int\limits_0^1(1-x)^\alpha x^\beta Jac^k_{\alpha,\beta}(x)Jac^l_{\alpha,\beta}(x)\\
=\frac{\Gamma(k+\alpha+1)\Gamma(k+\beta+1)}{(2k+\alpha+\beta+1)\Gamma(k+\alpha+\beta+1)k!}\delta_{kl}.
$$
\end{lemma}

\end{document}